\documentclass[12pt, reqno]{amsart}
\usepackage{amsmath, amsthm, amscd, amsfonts, amssymb, graphicx, xcolor}
\usepackage[bookmarksnumbered, colorlinks, plainpages]{hyperref}
\usepackage{mathrsfs}
\usepackage{tikz}
\usepackage{booktabs}
\usepackage{appendix}

\usetikzlibrary{arrows.meta, calc, decorations.pathreplacing, positioning}

\newtheorem{theorem}{Theorem}[section]
\newtheorem{lemma}[theorem]{Lemma}

\newtheorem{proposition}[theorem]{Proposition}
\newtheorem{corollary}[theorem]{Corollary}
\theoremstyle{definition}
\newtheorem{definition}[theorem]{Definition}
\newtheorem{example}[theorem]{Example}

\theoremstyle{remark}
\newtheorem{remark}[theorem]{Remark}
\numberwithin{equation}{section}

\newcommand{\unlinked}{\mathbin{\triangleright\!\triangleleft}}
\newcommand{\duto}{\stackrel{d_u}\rightarrow}

\begin{document}
\setcounter{page}{1}

\centerline{}

\centerline{}

\title[Geometrically Infinite Convergence Group Actions]{Topological Characterizations of Geometrically Infinite Convergence Group Actions and Orbit Uniform Metrics}

\author{Chaodong Yang}	
 \address{Beijing International Center for Mathematical Research\\
Peking University\\
 Beijing 100871, China
P.R.}
 \email{be2sio4@stu.pku.edu.cn}

\date{August 5th 2026}


\maketitle

\begin{abstract}
Two characterizations of geometric finiteness in terms of weaker conditions are known for actions on Gromov hyperbolic spaces. In this paper, we establish analogous characterizations for general convergence group actions. To this end, we introduce the notion of an orbit uniform metric. We prove that a point is either conical or bounded parabolic if and only if its orbit is discrete with respect to an orbit uniform metric. As a consequence, we characterize geometric infiniteness in terms of the uncountability of the set of non-conical limit points. We further characterize geometric infiniteness by the existence of escaping sequences of hyperbolic elements, thereby extending the corresponding result for actions on Gromov hyperbolic spaces.
\end{abstract}

\section{Introduction}
Suppose that a group $G$ admits a convergence action on a compact metrizable space $M$. The space $M$ naturally decomposes into the limit set $\Lambda G$ and the domain of discontinuity $\Omega G=M\setminus\Lambda G$. Such an action is said to be geometrically finite if every point of $\Lambda G$ is either conical or bounded parabolic. Otherwise, the action is said to be geometrically infinite.

The notion of geometric finiteness originated in the work of Ahlfors \cite{Ahl66} on Kleinian groups. Beardon and Maskit \cite{BM74} later proved that, for Kleinian groups, geometric finiteness is equivalent to the condition that every limit point is either conical or bounded parabolic. This dynamical characterization was subsequently extended by Bowditch \cite{Bow99} to proper isometric actions on Gromov hyperbolic spaces. In each of these settings, geometric finiteness has a natural geometric interpretation: there exists a $G$-invariant family of pairwise disjoint horoballs such that $G$ acts cocompactly on the complement of their union in the convex hull of the limit set. This geometric picture also underlies Bowditch's definition of relatively hyperbolic groups.

Although geometric finiteness was originally studied for actions on Gromov hyperbolic spaces, the notions of conical and bounded parabolic points depend only on boundary dynamics. Yaman \cite{Yam04} showed that every geometrically finite convergence action is induced by a geometrically finite action on a proper hyperbolic space. In contrast, for geometrically infinite actions it remains unknown whether every convergence action admits such a realization.

Geometric finiteness admits several characterizations in terms of conditions that appear substantially weaker than the defining condition. Beginning with the work of Bonahon \cite{Bon86} and continuing with subsequent work of Bishop \cite{Bis96}, Kapovich--Liu \cite{KL19}, and Yang--Yang \cite{YY26}, it has been shown that, for proper actions on Gromov hyperbolic spaces, geometric finiteness is equivalent to each of the following conditions: 
\begin{itemize}
    \item[(1)] there is no escaping sequence of hyperbolic elements;
    \item[(2)] the set of non-conical limit points is countable.
\end{itemize}

In the hyperbolic setting, where $G$ acts properly on a Gromov hyperbolic space $X$, a sequence $\{g_n\}$ of hyperbolic elements in $G$ is \textit{escaping} if
$$
\lim_{n\to\infty} d_X(\mathrm{Axis}(g_n), Go)=\infty,
$$
where $\mathrm{Axis}(g_n)$ is an axis of $g_n$.

Yang and Yang \cite{YY26} asked whether the above two characterizations extend to arbitrary convergence group actions. Theorems~\ref{geomInfiniteEscapingGeodesic} and~\ref{geomInfiniteUncountable} below answer this question affirmatively. 

We first introduce a boundary formulation of escaping sequences of hyperbolic elements, which is applicable without assuming the existence of an ambient Gromov hyperbolic space.
Given a convergence group action of $G$ on $M$, we call a sequence $\{g_n\}$ of hyperbolic elements in $G$ \textit{escaping} if 
$$
\lim_{n\to\infty} \sup_{h\in G}d(hg_n^+, hg_n^-)=0,
$$
where $d$ is a compatible metric on $M$, and $g_n^+, g_n^-$ are the attracting and repelling fixed points of $g_n$, respectively.

We prove the following two theorems.

\begin{theorem}\label{geomInfiniteEscapingGeodesic}
    Suppose that a group $G$ admits a non-elementary convergence action on a compact metrizable space $M$.
    Then the action is geometrically infinite if and only if there exists an escaping sequence of hyperbolic elements in $G$.
\end{theorem}

\begin{theorem}\label{geomInfiniteUncountable}
    Suppose that a group $G$ admits a non-elementary convergence action on a compact metrizable space $M$.
    Then the action is geometrically infinite if and only if the set of non-conical limit points is uncountable.
\end{theorem}

A fundamental result of Gerasimov \cite{Ger09} states that, for convergence group actions on arbitrary compacta, $2$-cocompactness is equivalent to geometric finiteness. To establish this result, Gerasimov developed the topological visibility theory and characterized conical and bounded parabolic points under the assumption of $2$-cocompactness. We observe that the proof of this characterization only requires the $2$-cocompactness of an orbit, rather than of the entire action. Consequently, the characterization extends to arbitrary convergence group actions, without assuming geometric finiteness.

This observation yields the following theorem. Given a set $X$, let $\Theta^2(X)$ be the collection of subsets of $X$ with cardinality $2$.

\begin{theorem}\label{topologicalVersion}
Let $G$ be a group admitting a topologically minimal non-elementary convergence action on a compactum $M$. Given $\xi\in M$, the following are equivalent:
\begin{itemize}
    \item[(1)] $\xi$ is either conical or bounded parabolic.
    \item[(2)] There exists a compact subset $K\subset\Theta^2(M)$ such that $\{\xi,\eta\}\in GK$ for every $\eta\in M\setminus\{\xi\}$.
    \item[(3)] There exists a compact subset $K\subset\Theta^2(M)$ such that $\Theta^2(G\xi)\subset GK$.
\end{itemize}
\end{theorem}

Although Theorem~\ref{topologicalVersion} holds for arbitrary compacta, all subsequent results are stated in the metrizable setting, where we define the orbit uniform metric. This metric serves as the main tool for deriving further geometric consequences, including the characterizations of geometrically infinite convergence actions in Theorems~\ref{geomInfiniteEscapingGeodesic} and \ref{geomInfiniteUncountable}. 

Given any compatible metric $d$ on $M$, the corresponding orbit uniform metric is defined by
$$
d_u(\xi,\eta)=\sup_{g\in G} d(g\xi,g\eta).
$$

The orbit uniform metric enables us to reformulate Theorem~\ref{topologicalVersion} in metric terms. More precisely, when $M$ is metrizable, the topological characterization above takes the following form.

\begin{corollary}\label{conicalOrBoundedParabolic}
    Suppose that a group $G$ admits a non-elementary convergence action on a compact space $M$.
    Given $\xi\in \Lambda G$, the following statements are equivalent:
    \begin{itemize}
        \item[(1)] $\xi$ is either conical or bounded parabolic.
        \item[(2)] $\xi$ is $d_u$-isolated in $\Lambda G$.
        \item[(3)] The orbit $G\xi$ is $d_u$-discrete.
    \end{itemize}
\end{corollary}

Since the non-discrete points of the orbit uniform metric are precisely the limit points which are neither conical nor bounded parabolic, the characterization of geometric finiteness by countability of non-conical limit points suggests studying the completeness of the orbit uniform metrics. Our next theorem shows that it is always complete.

\begin{theorem}\label{dGComplete}
    Suppose that a group $G$ admits a convergence action on a compact metrizable space $M$. Then every orbit uniform metric is complete.
\end{theorem}

Combining Corollary~\ref{conicalOrBoundedParabolic} and Theorem~\ref{dGComplete} provides a direct proof of Theorem~\ref{geomInfiniteUncountable}.

Although the overall strategy follows \cite{YY26}, the proof requires substantially different arguments because a general convergence action does not admit an ambient hyperbolic space. In particular, the translation-length techniques used in \cite{YY26} are no longer available. To overcome this, we introduce two notions of translation length adapted to convergence actions and establish the corresponding finiteness properties.

While Corollary~\ref{conicalOrBoundedParabolic} characterizes the local structure of the orbit uniform topology, it also leads to several global consequences. Beyond the characterizations above, the orbit uniform metric also has interesting intrinsic geometric properties. One example is the following density theorem.

\begin{theorem}\label{expansiveDense}
    The set of conical points and bounded parabolic points is $d_u$-dense in $\Lambda G$.
\end{theorem}

\subsection*{Structure of the paper}
The paper is organized as follows. 
In Section~\ref{chapter:convergence}, we recall the necessary background on convergence group actions. 
In Section~\ref{chapter:visibility}, we review Gerasimov's topological visibility theory and prove Theorem~\ref{topologicalVersion}.
In Section~\ref{chapter:orbit}, we introduce the orbit uniform metric, establish its basic properties, and prove Theorems~\ref{dGComplete} and~\ref{geomInfiniteUncountable}. 
In Section~\ref{chapter:escaping}, we introduce two notions of translation length for convergence group actions and establish their finiteness properties, from which we derive Theorem~\ref{geomInfiniteEscapingGeodesic}. 
In Section~\ref{chapter:global}, we establish some further properties of orbit uniform metrics.

\subsection*{Acknowledgments}
We thank Leonid Potyagailo and Wenyuan Yang for helpful discussions.

\section{Convergence Group Actions}\label{chapter:convergence}
In this section, we collect some basic facts about convergence group actions that will be used throughout the paper. Our presentation follows primarily \cite{Tuk94, Bow99}. We recall the definitions of convergence actions, limit sets, conical and parabolic points, and geometric finiteness, together with several standard properties that will be used later.

\begin{definition}
    Let $M$ be a compactum, and let $G$ be a group acting on $M$ by homeomorphisms. The action is called a \textit{convergence group action} if, for every sequence $\{g_n\}$ of distinct elements of $G$, there exist a subsequence $\{g_{n_i}\}$ and points $a, b \in M$ such that $g_{n_i} x \to b$ locally uniformly on $M \setminus \{a\}$.
\end{definition}

Equivalently, the action of $G$ on $M$ is a convergence group action if the induced action on the space
$$
\Theta^3(M):=\{(x, y, z)\in M^3: x, y, z \text{ are pairwise distinct.}\}
$$
is proper.

\begin{example}
The class of convergence group actions includes many natural examples.

\begin{itemize}
    \item[(1)] Suppose that $X$ is a proper Gromov hyperbolic space and that $G$ acts properly on $X$ by isometries. Then the action of $G$ on $X$ induces an action on the Gromov boundary $\partial X$, which is a convergence group action.
    \item[(2)] Let $G$ be a finitely generated group. Then the natural actions of $G$ on its Floyd boundary \cite{Kar03} and end boundary \cite{Sta68} are convergence group actions.
\end{itemize}
\end{example}

We next recall several standard notions associated with convergence group actions, including limit sets and the classification of elements.

\begin{definition}
    Let $\{g_n\}$ be a sequence in $G$ and let $a, b\in M$.
    If $g_n x\to b$ locally uniformly on $M\setminus\{a\}$, we say that $\{g_n\}$ is a \textit{convergence sequence} with \textit{attracting point} $b$ and \textit{repelling point} $a$.
\end{definition}

\begin{remark}
    If $\{g_n\}$ is a convergence sequence with attracting point $b$ and repelling point $a$, then $\{g_n^{-1}\}$ is a convergence sequence with attracting point $a$ and repelling point $b$.
\end{remark}

\begin{definition}
    For a convergence action of $G$ on $M$, the \textit{limit set} $\Lambda H$ of a subgroup $H \le G$ is the set of all attracting points of convergence sequences in $H$.

    The action is said to be \textit{non-elementary} if $\#\Lambda G > 2$.
\end{definition}

If the action is non-elementary, then the limit set $\Lambda G$ is the unique minimal nonempty closed $G$-invariant subset of $M$. Moreover, $\Lambda G$ is perfect (i.e., it has no isolated points) and is therefore uncountable. 

An element of $G$ is called \textit{elliptic} if it has finite order.
By the definition of a convergence group action, every element of infinite order has at least one and at most two fixed points in $M$.
Such an element is called \textit{hyperbolic} if it has exactly two fixed points and \textit{parabolic} if it has exactly one.

For a hyperbolic element $g\in G$, we denote its two fixed points in $M$ by $g^+$ and $g^-$, where $g^+$ is the attracting fixed point and $g^-$ is the repelling fixed point.

We next recall the dynamical classification of points in the limit set. Conical and parabolic points will play an important role in the definition of geometric finiteness.

\begin{definition}\label{defn:limitpoints}
    A point $\xi \in M$ is called \textit{conical} if there exists a convergence sequence $\{g_n\}$ with attracting point $b\in M$ and repelling point $\xi$, such that $\{g_n\xi\}$ converges to a point different from $b$.

    A point $\xi \in M$ is called \textit{parabolic} if its stabilizer $G_\xi$ is infinite and its limit set satisfies $\Lambda G_\xi = \{\xi\}$.

    A parabolic point $\xi$ is called \textit{bounded parabolic} if the action of $G_\xi$ on $M \setminus \{\xi\}$ is cocompact.
\end{definition}

Note that both conical and parabolic points belong to the limit set $\Lambda G$.
We denote by $\Lambda^{c} G$ and $\Lambda^{nc} G$ the sets of conical and non-conical points in $\Lambda G$, respectively. In particular, the fixed points of every hyperbolic element are conical.

The following lemma is a minor variant of \cite[Lemma~2C]{Tuk94}. For a proof under the present hypotheses, see the proof of \cite[Lemma~2.3]{YY26}.

\begin{lemma}\label{producingHyperbolicElements}
    Suppose that the action of $G$ on $M$ is a non-elementary convergence group action.
    If there exist an element $f \in G$ and an open set $U \subset M$ such that $f\overline{U} \subsetneqq U$, then $f$ is hyperbolic.
\end{lemma}

We use this result to prove the following lemma, which is implicit in the proof of \cite[Proposition~3.2]{YY26}.

\begin{figure}[h]
\centering
\begin{tikzpicture}[
    scale=1,
    point/.style={circle, fill=black, inner sep=1.2pt},
    boundarypoint/.style={circle, fill=red!70, inner sep=1.5pt},
    neighborhood/.style={line width=1.2pt},
    font=\small
]

\def\radius{5}
\def\startangle{50}
\def\endangle{130}

\draw[thick] (\startangle:\radius) arc (\startangle:\endangle:\radius);
\node at (40:5) {$M$};

\coordinate (xi) at (110:\radius);
\coordinate (gNxi) at (80:\radius);
\coordinate (eta) at (65:\radius);

\node[boundarypoint, label={[shift={(0,-0.6)}]below:$\xi$}] at (xi) {};
\node[boundarypoint, label={[shift={(0,-0.4)}]below:$g_{n_i}\xi$}] at (gNxi) {};
\node[boundarypoint, label=above:$\eta$] at (eta) {};

\def\rotangle#1{%
    \pgfmathparse{#1+90}%
    \pgfmathresult%
}

\pgfmathsetmacro{\xitheta}{110}
\pgfmathsetmacro{\xirot}{\xitheta+90}
\draw[neighborhood, blue!80, rotate around={\xirot:(xi)}] 
    (xi) ellipse (1.6 and 0.6);
\node[blue!80] at ($(xi)+(0.8,1.2)$) {$U$};

\draw[neighborhood, red!80, rotate around={\xirot:(xi)}] 
    (xi) ellipse (1.0 and 0.4);
\node[red!80] at ($(xi)+(-0.5,0.7)$) {$V$};

\pgfmathsetmacro{\gNtheta}{80}
\pgfmathsetmacro{\gNrot}{\gNtheta+90}
\draw[neighborhood, green!70!black, rotate around={\gNrot:(gNxi)}] 
    (gNxi) ellipse (0.9 and 0.4);
\node[green!70!black] at ($(gNxi)+(0.4,0.7)$) {$g_{n_i}U$};

\draw[neighborhood, orange!80, rotate around={\xirot:(xi)}] 
    (xi) ellipse (0.6 and 0.3);
\node[orange!80] at ($(xi)+(1.2,-0.3)$) {$g_{n_j} g_{n_i}U$};

\end{tikzpicture}
\caption{Illustrating the hyperbolic elements $g_ng_N$ in the proof of Proposition~\ref{lineNearOrbit}.}
\label{fig:neighborhoods}
\end{figure}
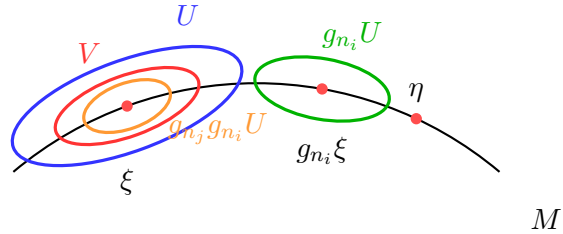

\begin{lemma}\label{producingShortHyperbolicElement}
    Let $\xi$ be a non-conical point in $M$.
    Suppose that a sequence $\{g_n\}$ in $G$ satisfies $g_n\xi\ne \xi$ for every $n\in \mathbb{N}$ and $\lim\limits_{n\to\infty}g_n\xi=\xi$. Then there exists a subsequence $\{g_{n_i}\}$ such that, for each fixed $i\in \mathbb{N}$, the element $g_{n_j}g_{n_i}$ is hyperbolic for all sufficiently large $j\in \mathbb{N}$.
\end{lemma}
\begin{proof}
    After passing to a subsequence, suppose that $\{g_{n_i}\}$ is a convergence sequence with attracting point $\zeta\in M$ and repelling point $\eta\in M$. We claim that $\zeta=\xi$. Indeed, if $\zeta\ne \xi$, then necessarily $\xi=\eta$ and $\lim\limits_{i\to\infty}g_{n_i}\xi=\xi\ne \zeta$. This would imply that $\xi$ is conical, a contradiction.

    By passing to a further subsequence, we may assume that $g_{n_i}\xi \ne \eta$ for every $i\in \mathbb{N}$.
    Fix $i\in \mathbb{N}$.
    Choose neighborhoods $U$ and $V$ of $\xi$ in $M$ such that $\xi \in V \subset \overline{V} \subsetneq U$ and $\eta \notin g_{n_i} \overline{U}$. Such neighborhoods exist because $M$ is a compact Hausdorff space with no isolated points and $g_{n_i} \xi \ne \eta$.
    Since $g_{n_j} x \to \xi$ locally uniformly on $M\setminus \{\eta\}$, $g_{n_j} (g_{n_i} \overline{U}) \subseteq V \subsetneq U$ for all sufficiently large $j\in \mathbb{N}$. See Figure~\ref{fig:neighborhoods}. By Lemma~\ref{producingHyperbolicElements}, the element $g_{n_j} g_{n_i}$ is hyperbolic.
\end{proof}

The preceding lemma provides a useful criterion for producing hyperbolic elements with prescribed dynamical properties. We next establish a convergence property of sequences of hyperbolic elements, whose proof relies on Lemma~\ref{producingShortHyperbolicElement}. This result will serve as a key ingredient in the study of the global structure of the $d_u$-metric in Section~\ref{chapter:global}.

\begin{lemma}\label{hyperbolicElementFixPointConvergence}
    Suppose that $\{g_n\}$ is a sequence of hyperbolic elements in $G$ such that $g_n^+\to \xi$ and $g_n^-\to \xi$ for some $\xi\in M$. Then $\{g_n\}$ is a convergence sequence whose attracting and repelling points are both $\xi$.
\end{lemma}

\begin{proof}
    Let $\{g_{n_i}\}$ be a convergence subsequence of $\{g_n\}$ with attracting point $a\in M$ and repelling point $b\in M$.
    It suffices to show that $a=b=\xi$.

    \textit{Case 1}: $a\ne b$.

    Let $U$ and $V$ be neighborhoods of $a$ and $b$, respectively, such that $\overline{U}\cap\overline{V}=\varnothing$.
    Then $g_n(M\setminus V)\subset U$ for sufficiently large $n$.
    Since $\overline{U}\subset M\setminus V$, this implies that $g_n\overline{U}\subset U$.
    Hence
    $$
    \{g_n^+\}=\bigcap_{k=1}^{\infty} g_n^k\overline{U}\subset U.
    $$
    Since $U$ is an arbitrary neighborhood of $a$, it follows that $g_n^+\to a$, i.e. $a=\xi$.
    By considering the sequence $\{g_n^{-1}\}$, we similarly obtain $b=\xi$, which is a contradiction.

    \textit{Case 2}: $a=b$.

    Suppose for contradiction that $a=b\ne \xi$.
    Choose a neighborhood $U$ of $\xi$ such that $a\notin \overline{U}$.
    Since $g_n^+\in U$ for sufficiently large $n\in \mathbb{N}$, the locally uniform convergence implies that $g_ng_n^+\to a$.
    However,
    $$
    g_ng_n^+=g_n^+\to \xi,
    $$
    which is a contradiction.
\end{proof}

We are now ready to recall the notion of geometric finiteness for convergence group actions.
This notion originates in the work of Beardon and Maskit \cite{BM74} on Kleinian groups.

\begin{definition}\label{defn:geomfiniteconvergence}
    A convergence group action of $G$ on $M$ is \textit{geometrically finite} if every point of $\Lambda G$ is either conical or bounded parabolic.
    The action is called \textit{geometrically infinite} if it is not geometrically finite.
\end{definition}

At this stage, the definition is purely dynamical: it refers only to the convergence group action and the types of points in its limit set. The connection between this formulation and the classical geometric setting is provided by the following theorem of Yaman \cite{Yam04}.

\begin{lemma}\label{YamanGeomFinite}
    Suppose that a group $G$ admits a geometrically finite convergence action on a compact metrizable space $M$.
    Then there exist a Gromov hyperbolic space $X$ and a geometrically finite proper action of $G$ on $X$ such that the Gromov boundary $\partial X$ is $G$-equivariantly homeomorphic to the limit set $\Lambda G\subset M$.
\end{lemma}

We next introduce the $2$-cocompactness.
Let $M$ be a compactum and put $S^2(M)=M^2/\sim$, where $(x, y)\sim(y, x)$ for every $x, y\in M$.
We denote by $\{x, y\}$ the unordered pair in $S^2(M)$, that is, the equivalence class of $(x, y)$ in $M^2/\sim$.
Let $\Delta(M)=\{(x, x)\in S^2(M): x\in M\}$ be the diagonal in $S^2(M)$ and $\Theta^2(M)=\{\{x, y\}\in S^2(M): x\ne y\}$.
Then $S^2(M)=\Theta^2(M)\sqcup\Delta(M)$ is a compactification of $\Theta^2(M)$.

The following lemma, combining results of Tukia \cite{Tuk98} and Gerasimov
\cite{Ger09}, characterizes geometric finiteness in terms of $2$-cocompactness. The group action of $G$ on $M$ is called $2$-cocompact if the induced action of $G$ on $\Theta^2(M)$ is cocompact.

\begin{lemma}\label{2cocompact}
    Suppose that a group $G$ admits a convergence action on a compactum $M$. Then the action of $G$ on $M$ is $2$-cocompact if and only if it is geometrically finite.
\end{lemma}

\section{Topological Visibility Theory}\label{chapter:visibility}
In this section, we review the topological visibility theory developed by Gerasimov \cite{Ger09}, which will be used to study the geometry of convergence group actions. We recall the notions of entourages, shadows, and horospheres and establish several properties needed in the proof of Theorem~\ref{topologicalVersion}.

\subsection{Entourages}

\begin{definition}
    A neighborhood of the diagonal $\Delta(M)$ in $S^2(M)$ is called an \textit{entourage} of $M$.
    We denote by $\mathrm{Ent}(M)$ the space of entourages of $M$.
\end{definition}

\begin{remark}
    Here, by a neighborhood of $\Delta(M)$, we mean a subset of $S^2(M)$ that contains an open set $U$ with $\Delta(M)\subset U$; the neighborhood itself is not required to be open.
\end{remark}

An entourage $\textbf{u}$ of $M$ defines a graph structure on $M$ as follows.
The vertex set is $M$, and two points $x, y\in M$ are joined if and only if $\{x, y\}\in \textbf{u}$.
We denote by $\delta_{\textbf{u}}$ the corresponding graph metric.

Set $\textbf{u}^n=\{\{x, y\}\in S^2(M): \delta_{\textbf{u}}(x, y)\le n\}$.
Then $\textbf{u}^n$ is also an entourage.
It follows from the Lebesgue covering theorem that, for every $\textbf{u}\in \mathrm{Ent}(M)$ and $n\in \mathbb{N}$, there exists $\textbf{u}\in \mathrm{Ent}(M)$ such that $\textbf{u}^n\subset \textbf{u}$.

\begin{definition}
    Let $\textbf{u}\in \mathrm{Ent}(M)$ be an entourage. A subset $S\subset M$ is called $\textbf{u}$-\textit{small} if its diameter with respect to the metric $\delta_\textbf{u}$ is $0$ or $1$.
\end{definition}

\begin{definition}
    Two entourages $\textbf{u}, \textbf{v}\in \mathrm{Ent}(M)$ are called \textit{unlinked} if $M=A\cup B$ where $A$ is $\textbf{u}$-small and $B$ is $\textbf{v}$-small, denoted by $\textbf{u}\unlinked \textbf{v}$.
    Two entourages $\textbf{u}, \textbf{v}\in \mathrm{Ent}(M)$ are called \textit{linked} if they are not unlinked, denoted by $\textbf{u}\# \textbf{v}$.
\end{definition}

From now on, we only consider self-linked entourages.

\begin{definition}
    Let $\textbf{u}, \textbf{v}\in \mathrm{Ent}(M)$ be entourages. If $\textbf{u}\unlinked \textbf{v}$, we define the \textit{shadow sets}
    $$\mathrm{Sh}_{\textbf{u}}\textbf{v}=\{S\subset M: S\text{ is } \textbf{u}\text{-small and }M-S\text{ is }\textbf{v}\text{-small}\}$$
    and $\mathrm{sh}_{\textbf{u}}\textbf{v}=\cap\mathrm{Sh}_{\textbf{u}}\textbf{v}$.

    Let $\textbf{u}\in \mathrm{Ent}(M)$ be an entourage and let $\xi\in M$. We define $\mathrm{sh}_{\textbf{u}}\xi=\{\xi\}$.
\end{definition}

The above constructions allow us to regard entourages as points representing regions in the internal geometry. We therefore enlarge $M$ by adjoining all entourages and endow the resulting space with a topology.

There is a topology on $\tilde{M}:=M\sqcup\mathrm{Ent}(M)$ defined as follows.
For every open subset $U$ of $M$, the corresponding subset
$$
\tilde{U}=U\cup \{\textbf{u}\in \mathrm{Ent}(M): M-U \text{ is }\textbf{u}\text{-small}\}
$$
is open in $\tilde{M}$.
The topology on $\tilde{M}$ is generated by all subsets of $\mathrm{Ent}(M)$ together with all sets $\tilde{U}$, where $U$ ranges over the open subsets of $M$.

\begin{definition}
    Suppose that $\textbf{b}$ is an entourage of $M$ and that $\textbf{a}, \textbf{c}\in \tilde{M}$ satisfy $\textbf{a}\unlinked \textbf{b}$ and $\textbf{b}\unlinked \textbf{c}$ whenever these conditions are required.
    Given $k\in \mathbb{N}$, we say that $\textbf{b}$ $k$-\textit{lies between} $\textbf{a}$ and $\textbf{c}$ if $\delta_{\textbf{b}}(\mathrm{sh}_{\textbf{b}}\textbf{a}, \mathrm{sh}_{\textbf{b}}\textbf{c})>k$, denoted by $\textbf{a}-\textbf{b}-\textbf{c}(k)$.
\end{definition}

\begin{lemma}\cite[Lemma~T2]{Ger09}\label{T2}
    Suppose that $k\in \mathbb{N}$, that $\textbf{a}, \textbf{d}\in \tilde{M}$, and that $\textbf{b}, \textbf{c}\in \mathrm{Ent}(M)$. If $\textbf{a}-\textbf{b}-\textbf{d}(k+1)$ and $\textbf{b}-\textbf{c}-\textbf{d}(2)$, then $\textbf{a}-\textbf{c}-\textbf{d}(k)$.
\end{lemma}

\subsection{Horoballs}
To prove the implication $(3)\implies(1)$ in Corollary~\ref{conicalOrBoundedParabolic}, we assume the separation condition stated at the beginning of this subsection. The construction of horospheres and the underlying ideas follow those in \cite{Ger09}. The main difference is that the global $2$-cocompactness assumption in \cite{Ger09} is replaced here by that for the orbit of a non-conical point. Consequently, the key properties of horospheres and cusp regions must be established under this weaker hypothesis.

Suppose that a group $G$ admits a topologically minimal convergence action on $M$, i.e. $M=\Lambda G$.
We fix a non-conical point $\xi\in \Lambda^{nc}G$ and assume that there exists a compact subset $K$ of $\Theta^2(M)$ such that $\Theta^2(G\xi)\subset GK$.

Fix $k\ge 3$ and let $\textbf{e}$ be a self-linked entourage such that $\textbf{e}^{2k+1}\subset \Theta^2(M)\setminus K$.
Let $A=G\textbf{e}\subset \mathrm{Ent}(M)$.

\begin{definition}
    A set of entourages is \textit{bounded} if every element is linked with a fixed entourage.

    A set of entourages is \textit{discrete in} $\mathrm{Ent}(M)$ if every bounded subset is finite.
\end{definition}

Here, boundedness and discreteness are understood with respect to the visibility structure rather than a metric on $M$.

\begin{lemma}\cite[Proposition~P]{Ger09}\label{orbitDiscrete}
    If the action of $G$ on $M$ is a convergence group action, then every orbit in $\mathrm{Ent}(M)$ is discrete.
\end{lemma}

This property allows us to control the number of entourage orbits that interact with a fixed bounded region.
We now use these notions to describe the horosphere associated with a non-conical point.

\begin{definition}
    For $\xi\in M$ and $k\in \mathbb{N}$, we define the \textit{horosphere}
$$
H_{k}(\xi)=\{\textbf{u}\in A: \text{there is no }\textbf{v}\in A \text{ such that } \textbf{u}-\textbf{v}-\xi(k)\}.
$$

    A point $\xi\in M$ is called $k$-\textit{conical} if $H_{k}(\xi)=\varnothing$.
\end{definition}

\begin{lemma}\cite[Lemma~C2]{Ger09}\label{C2}
    Every $2$-conical point is conical in the sense of convergence group actions.
\end{lemma}

In particular, $\xi$ is non-$k$-conical for every $k\ge 2$.

For an entourage $\textbf{u}\in A$ and $k\in \mathbb{N}$, we denote by
$$
P_{k}(\textbf{u})=\{\xi\in M: \textbf{u}\in H_{k}(\xi)\}.
$$

\begin{lemma}\cite[Lemma~H1]{Ger09}\label{H1}
    Let $\xi\in M$ be a non-$k$-conical point. For every entourage $\textbf{u}$, either $\textbf{u}\in H_{k}(\xi)$ or there exists $\textbf{v}\in H_{k}(\xi)$ such that $\textbf{u}-\textbf{v}-\xi(k)$.
\end{lemma}

\begin{lemma}\cite[Lemma~H2]{Ger09}\label{H2}
    If $\textbf{a}\in A$ and $\textbf{u}\unlinked \textbf{a}$, then $\operatorname{diam}_{\delta_\textbf{a}}(P_{k}(\textbf{u}))\le 2k+1$ for every $k\in \mathbb{N}$.
\end{lemma}

The next lemma is the key separation property of the orbit $G\xi$. It shows that distinct orbit points can be separated by an entourage lying between them.

\begin{lemma}\label{separatingOrbit}
    For any $g, h\in G$ with $g\xi\ne h\xi$, there exists $\textbf{u}\in A$ such that $g\xi-\textbf{u}-h\xi(2k+1)$.
\end{lemma}
\begin{proof}
    By our assumption, there exists $f\in G$ such that $\{g\xi, h\xi\}\in fK$.
    Hence $\delta_{f^{-1}\textbf{e}}(g\xi, h\xi)>2k+1$.
    Therefore, the entourage $\textbf{u}=f^{-1}\textbf{e}\in A$ has the required property.
\end{proof}

The preceding separation lemma implies the following finiteness property, which will allow us to obtain finitely many horosphere orbits.

\begin{lemma}\label{finitenessHoroball}
    For every $\textbf{u}\in A$, the intersection $P_{k}(\textbf{u})\cap G\xi$ is finite.
\end{lemma}
\begin{proof}
    Suppose that $g\xi\ne h\xi$ are points in $P_{k}(\textbf{u})\cap G\xi$.
    By Lemma~\ref{separatingOrbit}, there exists $\textbf{v}\in A$ such that $g\xi-\textbf{v}-h\xi(2k+1)$.
    Lemma~\ref{H2} implies that $\textbf{v}\#\textbf{u}$.
    By Lemma~\ref{orbitDiscrete}, the orbit $A$ is discrete in $\mathrm{Ent}(M)$, and hence there are only finitely many such $\textbf{v}$.
    This implies that $P_{k}(\textbf{u})\cap G\xi$ is discrete in the compact space $M$, and hence finite.
\end{proof}

\subsection{Proof of Theorem~\ref{topologicalVersion}}
We now prove Theorem~\ref{topologicalVersion}. The implications $(1)\implies(2)$ and $(2)\implies(3)$ are straightforward, whereas the converse implication $(3)\implies(1)$ is the main part of the proof and follows from the horoball construction developed in the previous subsection.
\\
    \paragraph{\textit{(1)}$\implies$\textit{(2)}.}
    Suppose $\xi$ is a conical point in $M$.
    Then there exists a sequence $\{g_n\}$ in $G$ and points $a\ne b\in M$ such that $g_n(M-\{\xi\})\to a$ and $g_n\xi\to b$.
    Choose a compact neighborhood $K$ of $\{a, b\}$ in $\Theta^2(M)$.
    Hence for every $\eta\in M\setminus\{\xi\}$, we have $\{g_n\xi, g_n\eta\}\in K$ for all sufficiently large $n\in \mathbb{N}$.
    Consequently, $\{\xi, \eta\}\in GK$ for every $\eta\in M\setminus\{\xi\}$.

    Suppose $\xi$ is a bounded parabolic point in $M$.
    Then the action of the stabilizer $Stab_G(\xi)$ on $M\setminus\{\xi\}$ is cocompact.
    Choose a compact subset $D$ of $M\setminus\{\xi\}$ such that $Stab_G(\xi)\cdot D=M\setminus\{\xi\}$.
    Let $K=\{\{\xi, \zeta\}\in \Theta^2(M): \zeta\in D\}$.
    Then $K$ is compact and $\{\xi, \eta\}\in Stab_G(\xi)\cdot K\subset GK$.
\\
    \paragraph{\textit{(2)}$\implies$\textit{(3)}.}
    This follows directly from the fact that $G$ acts on $M$ by homeomorphisms.
\\
    \paragraph{\textit{(3)}$\implies$\textit{(1)}.}
    Suppose $\xi\in \Lambda^{nc}G$ and the orbit $G\xi$ is discrete with respect to $d_u$.
    Consider the set
    $$X=\{(\textbf{u}, g\xi)\in A\times G\xi: \textbf{u}\in H_{k}(g\xi)\}.$$
    The group $G$ acts on $X$ by $h(\textbf{u}, g\xi)=(h\textbf{u}, hg\xi)$.
    On the one hand, by Lemma~\ref{finitenessHoroball}, $X/G=\{[(\textbf{e}, g\xi)]: g\xi\in P_{k}(\textbf{e})\}$ is finite.
    On the other hand,
    $X/G=\{[(\textbf{u}, \xi)]: \textbf{u}\in H_{k}(\xi)\}/Stab_{G}(\xi)$.
    Therefore, $H_{k}(\xi)/Stab_{G}(\xi)$ is finite.

    Choose $\textbf{u}_1, \cdots, \textbf{u}_{m}\in H_{k}(\xi)$ representing the elements of $H_{k}(\xi)/Stab_{G}(\xi)$.
    Let
    $$K=\{\eta\in M: \{\xi, \eta\}\notin \textbf{u}_{i} \text{ for some }1\le i\le m\}.$$

    It is clear that $\overline{K}\subset M\setminus\{\xi\}$.
    We will show that $Stab_G(\xi)\overline{K}=M\setminus\{\xi\}$.

    For any $g\xi\ne \xi$ in $G\xi$, by Lemma~\ref{separatingOrbit} there exists $\textbf{u}\in A$ such that $\xi-\textbf{u}-g\xi(2k+1)$.
    By Lemma~\ref{H1} and Lemma~\ref{T2}, there exists $\textbf{v}\in H_{k}(\xi)$ such that $\xi-\textbf{v}-g\xi(k)$.
    Hence there exist $1\le i\le m$ and $h\in Stab_{G}(\xi)$ such that $\textbf{v}=h\textbf{u}_i$.
    Thus, $(\xi=)h^{-1}\xi-\textbf{u}_i-h^{-1}g\xi(k)$.
    By the definition of betweenness, we have $h^{-1}g\xi\in K$, i.e. $g\xi\in hK$.
    The argument above shows that $G\xi\setminus\{\xi\}\subset Stab_{G}(\xi)K$.

    Since the action is topologically minimal, the orbit $G\xi$ is dense in $M$. 
    It follows that $\overline{Stab_G(\xi)K}$ is equal to $M$.
    Hence $Stab_G(\xi)$ is infinite.
    Since $\xi$ is non-conical, the limit set $\Lambda Stab_{G}(\xi)=\{\xi\}$, i.e. $\xi$ is a parabolic point.
    (Otherwise, the group $Stab_{G}(\xi)$ would contain hyperbolic elements, a contradiction.)
    Hence the action of $Stab_{G}(\xi)$ on $M\setminus\{\xi\}$ is properly discontinuous.
    Therefore, $Stab_{G}(\xi)\overline{K}=M\setminus \{\xi\}$, i.e. $\xi$ is bounded parabolic.
\\
\paragraph{\textbf{Notation}} We denote by
$$\Lambda^{e}G:=\Lambda^{c}G\cup \Lambda^{bp}G$$
the set of \textit{expansive points}, and by
$$\Lambda^{ne}G:=\Lambda G\setminus \Lambda^{e}G$$
the set of \textit{non-expansive points}. By Theorem~\ref{topologicalVersion}, expansive points are precisely the points on which $G$ acts expansively.

\section{The Orbit Uniform Metric}\label{chapter:orbit}
From now on, we assume that a group $G$ admits a convergence action on a compact metrizable space $M$. 
In this section, we introduce the main construction of this paper: the orbit uniform metrics associated with a convergence group action. These metrics endow $M$ with an intrinsic metric structure determined by the group action and serve as a fundamental tool throughout the paper.

\begin{definition}
    Given a compatible metric $d$ on $M$, the corresponding \textit{orbit uniform metric} on $M$ is defined by
    $$
    d_{u}(\xi, \eta)=\sup_{g\in G} d(g\xi, g\eta).
    $$
\end{definition}

Note that the supremum is finite since $M$ is compact.
We first verify that $d_u$ is indeed a metric.

\begin{proposition}
    The function $d_u$ is a metric on $M$.
\end{proposition}
\begin{proof}
    It is clear that, for every $\xi, \eta\in M$, we have $d_u(\xi, \eta)=0$ if and only if $\xi=\eta$, and that $d_u(\xi, \eta)=d_u(\eta, \xi)$.
    
    Let $\xi, \eta, \zeta$ be points in $M$.
    For any $\varepsilon>0$, there exists $g\in G$ such that $d(g\xi, g\zeta)\ge d_u(\xi, \zeta)-\varepsilon$.
    Then
    $$
        d_u(\xi, \zeta)\le d(g\xi, g\zeta)+\varepsilon\le d(g\xi, g\eta)+d(g\eta, g\zeta)+\varepsilon\le d_u(\xi, \eta)+d_u(\eta, \zeta)+\varepsilon.
    $$
    Since $\varepsilon>0$ is arbitrary, it follows that
    $$
    d_u(\xi, \zeta)\le d_u(\xi, \eta)+d_u(\eta, \zeta).
    $$
\end{proof}

In general, the topology induced by an orbit uniform metric differs from the original topology on $M$. Thus, an orbit uniform metric provides a genuinely new metric structure on $M$, rather than merely reformulating the original metric $d$.
Note that the action of $G$ on $(M, d_u)$ is isometric and $d\le d_u$.
\\
\medskip
\paragraph{\textbf{Notation}} Unless otherwise specified, all topological and metric notions refer to the original topology on $M$ and the compatible metric $d$, respectively. 

We now show that the topology induced by the orbit uniform metric is intrinsic to the convergence action.

\begin{proposition}
    The topology induced by $d_u$ is independent of the choice of the compatible metric $d$.
    Moreover, the associated uniform structure is also independent of this choice, i.e., if $d$ and $d^\prime$ are two compatible metrics on $M$, then $\mathrm{id}_M:(M,d_u)\to(M,d_u^\prime)$ is uniformly continuous.
\end{proposition}

\begin{proof}
    Let $d$ and $d^\prime$ be two compatible metrics on $M$.
    Then the identity map $\mathrm{id}_M:(M,d)\to(M,d^\prime)$ is continuous.
    Since $M$ is compact, the map is uniformly continuous, i.e., for every $\varepsilon>0$ there exists $\delta>0$ such that $d(\xi,\eta)<\delta$ implies $d^\prime(\xi,\eta)<\varepsilon$.

    Indeed, if $d_u(\xi,\eta)<\delta$, then $d(g\xi,g\eta)<\delta$ for every $g\in G$. Hence $d^\prime(g\xi,g\eta)<\varepsilon$ for every $g\in G$, and therefore $d_u^\prime(\xi,\eta)<\varepsilon$.

    By interchanging the roles of $d$ and $d^\prime$, we conclude that the two metrics are uniformly equivalent.
\end{proof}

In the metric setting, we can state Theorem~\ref{topologicalVersion} as Corollary~\ref{conicalOrBoundedParabolic}.

\begin{theorem}[=Theorem~\ref{dGComplete}]
    Every orbit uniform metric is complete.
\end{theorem}
\begin{proof}
    Let $\{\xi_n\}$ be a $d_u$-Cauchy sequence in $M$.
    Since $d_u\ge d$, it is also a Cauchy sequence with respect to $d$.
    Since $M$ is compact, the sequence converges to some point $\xi\in M$.
    For any $\varepsilon>0$, there exists $N>0$ such that $d_u(\xi_n, \xi_m)<\varepsilon$ for every $n, m>N$.
    Note that $\lim\limits_{m\to \infty}g\xi_m=g\xi$ for every $g\in G$.
    Hence
    $$d(g\xi_n, g\xi)=\lim_{m\to\infty}d(g\xi_n, g\xi_m)\le \varepsilon$$
    for every $g\in G$ and every $n>N$.
    Therefore, for every $\varepsilon>0$, there exists $N>0$ such that $d_u(\xi_n, \xi)\le \varepsilon$ for every $n>N$; that is, $\{\xi_n\}$ $d_u$-converges to $\xi$.
\end{proof}

As an immediate consequence of Corollary~\ref{conicalOrBoundedParabolic} and Theorem~\ref{dGComplete}, we obtain Theorem~\ref{geomInfiniteUncountable}.

\begin{corollary}[=Theorem~\ref{geomInfiniteUncountable}]
    The action of $G$ on $M$ is geometrically infinite if and only if the set of non-conical limit points is uncountable.
\end{corollary}
\begin{proof}
\textit{Necessity direction}. By definition, if the action is geometrically finite, then every non-conical point is a bounded parabolic point. It is clear that the set of bounded parabolic points is countable.
\\
\textit{Sufficiency direction}. Suppose that the action is geometrically infinite. By definition, the set $\Lambda^{ne}G$ is non-empty.
By Corollary~\ref{conicalOrBoundedParabolic}, every point $\xi$ in $\Lambda^{ne}G$ is not $d_u$-isolated in $G\xi\subset \Lambda^{ne}G$, i.e. the space $\Lambda^{ne}G$ is $d_u$-perfect.
By Theorem~\ref{dGComplete}, $(\Lambda^{ne}G, d_u)$ is a perfect complete metric space, and hence its cardinality is at least that of the continuum.
\end{proof}

\begin{remark}
    The spaces $M$ and $\Lambda^{ne}G$ are not $d_u$-compact. 
    
    Since $M$ contains infinitely many conical points, which are $d_u$-isolated, the $d_u$-topology of $M$ is not compact.
    
    Consider $\xi\in \Lambda^{ne}G$ and a sequence $\{g_n\}$ in $G$ such that $g_n\xi\to \eta\in \Lambda^{e}G$.
    Such a sequence exists since the orbit $G\xi$ is dense in $\Lambda G$.
    Then the sequence $\{g_n\xi\}$ has no $d_u$-convergent subsequence.
\end{remark}

\section{Escaping Hyperbolic Elements}\label{chapter:escaping}
In this section, we prove Theorem~\ref{geomInfiniteEscapingGeodesic}. The overall strategy follows that of \cite{YY26}. However, unlike the setting of Gromov hyperbolic spaces considered there, a convergence group action does not come equipped with an ambient geometry. Consequently, several geometric notions used in \cite{YY26} are no longer available and must be replaced by new concepts adapted to the present setting.

The main point is to prove the converse direction: if there is no escaping sequence of hyperbolic elements, then the action is geometrically finite. To this end, we first establish the implication under the following non-escaping assumption and then complete the proof of the theorem in Subsection~\ref{subsec:proofTheorem1}.

The following condition captures the absence of escaping hyperbolic elements in terms of the orbit uniform metric.
\begin{definition}
    Given a subset $A$ of $G$ and $\varepsilon>0$, we say that $A$ has the $\varepsilon$-\textit{NEHE property} (Non-Escaping Hyperbolic Elements) if $d_{u}(g^+, g^-)>\varepsilon$ for every hyperbolic element $g\in A$.
\end{definition}

\subsection{Translation lengths}
In this subsection, we introduce two notions of translation length for hyperbolic elements in the boundary setting and investigate their fundamental properties. These constructions constitute the central step of the proof: they replace the role of geometric translation length in the absence of an ambient hyperbolic space and allow us to recover the necessary quantitative estimates from the orbit uniform metric.

\begin{definition}
    Let $g\in G$ be a hyperbolic element. Its \textit{topological translation length} is defined by
    $$
    \tau(g):=\inf_{\xi\in M\setminus\{g^-, g^+\}} \sup_{n\in \mathbb{Z}} d(g^n\xi, g^{n+1}\xi).
    $$
    Its $d_u$-\textit{translation length} is defined by
    $$
    \tau_u(g):=\inf_{\xi\in M\setminus\{g^+, g^-\}} d_u(\xi, g\xi).
    $$

    Let $c$ be a hyperbolic conjugacy class in $G$ (i.e., consisting of hyperbolic elements). Its \textit{topological translation length} is defined by
    $$
    \tau(c):=\sup_{g\in c}\tau(g).
    $$
\end{definition}

The first notion reflects the dynamical behavior of a hyperbolic element with respect to the original topology, whereas the second measures its displacement in the orbit uniform metric. The comparison between these two quantities allows us to transfer dynamical information into estimates in the orbit uniform metric.

\begin{remark}
    The topological translation length is not invariant under conjugation, whereas the $d_u$-translation length is.
\end{remark}

\begin{lemma}\label{threeTranslationLengths}
    $\tau(g)\le \tau([g])\le \tau_u(g)$ for every hyperbolic element $g\in G$.
\end{lemma}
\begin{proof}
    The inequality $\tau(g)\le \tau([g])$ follows directly from the definition.

    For every $h\in G$, since $(hgh^{-1})^\pm=hg^\pm$, it follows that
    \begin{align*}
        \tau(hgh^{-1})&=\inf_{\xi\in M\setminus\{hg^+, hg^-\}}\sup_{n\in \mathbb{Z}} d(hg^nh^{-1}\xi, hg^{n+1}h^{-1}\xi)\\
        &=\inf_{\eta=h^{-1}\xi\in M\setminus\{g^+, g^-\}}\sup_{n\in \mathbb{Z}} d(hg^{n}\eta, hg^{n+1}\eta)\\
        &\le \inf_{\eta\in M\setminus\{g^+, g^-\}}d_{u}(\eta, g\eta)\\
        &=\tau_u(g).
    \end{align*}

    Hence $\tau([g])\le \tau_u(g)$.
\end{proof}

Before studying small translation lengths, we first show that every hyperbolic element has positive topological translation length.

\begin{lemma}\label{translationLengthShort}
    Let $g\in G$ be a hyperbolic element. Then $\tau(g)>0$.
\end{lemma}

\begin{proof}
    Note that the action of $\langle g\rangle$ on $M\setminus\{g^-, g^+\}$ is cocompact. Hence we may choose a compact subset $K\subset M\setminus\{g^-, g^+\}$ such that $\langle g\rangle K=M\setminus\{g^-, g^+\}$.

    Consider the continuous function $\varphi: K\to \mathbb{R}_+$ defined by $\xi\mapsto d(\xi, g\xi)$.
    Since $K$ is compact and $\langle g\rangle K=M\setminus\{g^-, g^+\}$, we obtain $\tau(g)\ge \min\varphi>0$.
\end{proof}

The next lemma provides the key finiteness property needed for the proof. It shows that, under the NEHE assumption, hyperbolic elements with small translation lengths can occur in only finitely many conjugacy classes.

\begin{lemma}
    Suppose that $A$ is a subset of $G$ with the $\varepsilon$-NEHE property for some $\varepsilon>0$. Then there exists a constant $r=r(\varepsilon)>0$ such that there are only finitely many hyperbolic conjugacy classes $c$ with $c\cap A\ne \varnothing$ and $\tau(c)\le r$.
\end{lemma}

\begin{proof}
    Set $r=\frac{1}{20}\varepsilon$.
    We claim that $r$ satisfies the required property.

    Let $c$ be a hyperbolic conjugacy class intersecting $A$.
    By definition, we have $\tau(g)\le r$ for every $g\in c\cap A$.
    Since $(hgh^{-1})^\pm=hg^\pm$, the NEHE property implies that there exists $g\in G$ such that $\tau(g)\le r$ and $d(g^+, g^-)\ge \varepsilon$.

    Define
    \begin{align}
        B&:=\{g\in G: \tau(g)\le r\text{ and }d(g^+, g^-)\ge \varepsilon\},\notag\\
        C&:=\{\text{hyperbolic conjugacy classes } c: \tau(c)\le r\text{ and }c\cap A\ne \varnothing\}.\notag
    \end{align}
    The discussion above gives an injective map $C\to B$.

    Let $g\in B$. Choose $\xi\in M\setminus\{g^-, g^+\}$ such that
    $$
    \sup_{n\in \mathbb{Z}}d(g^n\xi, g^{n+1}\xi)\le \inf_{\eta\in M\setminus\{g^-, g^+\}} \sup_{n\in \mathbb{Z}}d(g^n\eta, g^{n+1}\eta)+r\le 2r.
    $$

    Since $\lim\limits_{n\to\infty}g^n\xi=g^+$, $\lim\limits_{n\to -\infty} g^n\xi=g^-$, and $d(g^+, g^-)\ge \varepsilon$, there exists $n\in \mathbb{Z}$ such that $d(g^n\xi, g^-)\ge \frac{1}{4}\varepsilon$, $d(g^n\xi, g^+)\ge \frac{1}{4}\varepsilon$, $d(g^{n+1}\xi, g^-)\ge \frac{1}{4}\varepsilon$, and $d(g^{n+1}\xi, g^+)\ge \frac{1}{4}\varepsilon$.

    Set
    $$
    X=\{(x,y,z)\in M^3:
    d(x,z)\ge \varepsilon,\ 
    d(x,y)\ge \tfrac{1}{4}\varepsilon,\ 
    d(y,z)\ge \tfrac{1}{4}\varepsilon\}.
    $$
    Then $X$ is a compact subset of $\Theta^3(M)$.

    The discussion above shows that the triples $(g^-,g^n \xi,g^+)$ and $g(g^-,g^n \xi,g^+)=(g^-,g^{n+1}\xi,g^+)$ lie in $X$ for every $g\in B$. Hence $gX\cap X\neq\varnothing$ for every $g\in B$.

    Since the action of $G$ on $\Theta^3(M)$ is proper, the set $B$ is finite. Consequently, the set $C$ is finite as well.
\end{proof}

\begin{corollary}\label{translationLengthLowerBound}
    Suppose that $A$ is a subset of $G$ with the $\varepsilon$-NEHE property for some $\varepsilon>0$.
    Then there exists a constant $r^\prime>0$ such that $\tau(c)>r^\prime$ for every hyperbolic conjugacy class $c$ intersecting $A$.

    In particular, $\tau_u(g)>r^\prime$ for every hyperbolic element $g\in A$.
\end{corollary}

\begin{lemma}\label{lineNearOrbit}
    Suppose that $G$ has the $\varepsilon$-NEHE property for some $\varepsilon>0$.
    Then for every non-conical point $\xi$, the orbit $G\xi$ is $d_u$-discrete, i.e. there exists a constant $r_\xi>0$, depending on $\xi$, such that
    $$d_u(\xi, g\xi)\ge r_\xi$$
    for every $g\in G$ with $g\xi\ne \xi$.
\end{lemma}

\begin{proof}
    By Corollary~\ref{translationLengthLowerBound}, there exists a constant $r^\prime>0$ such that $\tau_u(g)>r^\prime$ for every hyperbolic element $g\in G$.
    By the definition of the $d_u$-translation length, we have $d_u(g\xi, \xi)>r^\prime$ for every hyperbolic element $g\in G$.

    Suppose, on the contrary, that for every $n\in \mathbb{N}$ there exists an element $g_n\in G$ such that $g_n\xi\ne \xi$ and $d_u(\xi, g_n\xi)\to 0$.
    In particular, $\lim\limits_{n\to\infty}g_n\xi=\xi$.

    Lemma~\ref{producingShortHyperbolicElement} gives a subsequence $\{g_{n_i}\}$ of $\{g_n\}$.
    Choose sufficiently large $i\in \mathbb{N}$ such that $d_u(\xi, g_{n_i}\xi)<\frac{1}{2}r^\prime$, and then choose sufficiently large $j\in \mathbb{N}$ such that $d_u(\xi, g_{n_j}\xi)<\frac{1}{2}r^\prime$ and the product $g_{n_j}g_{n_i}$ is hyperbolic.
    By Corollary~\ref{translationLengthLowerBound}, we have
    $$
    d_u(\xi, g_{n_j}g_{n_i}\xi)\ge r^\prime.
    $$
    However,
    \begin{align}
        d_u(\xi, g_{n_j}g_{n_i}\xi)&\le d_u(\xi, g_{n_j}\xi)+d_u(g_{n_j}\xi, g_{n_j}g_{n_i}\xi)\notag\\
        &=d_u(\xi, g_{n_j}\xi)+d_u(\xi, g_{n_i}\xi)\notag\\
        &< \frac{1}{2}r^\prime+\frac{1}{2}r^\prime\notag\\
        &=r^\prime,\notag
    \end{align}
    which is a contradiction.
\end{proof}

\subsection{Proof of Theorem~\ref{geomInfiniteEscapingGeodesic}}\label{subsec:proofTheorem1}
With the preceding estimates established, we now prove Theorem~\ref{geomInfiniteEscapingGeodesic}.
\\
    \paragraph{\textit{Necessity direction}.}
By Lemma~\ref{2cocompact}, the action of $G$ on $M$ is $2$-cocompact. Let $K$ be a compact subset of $\Theta^2(M)$ such that $GK=\Theta^2(M)$. Consider the function $\delta: K\to \mathbb{R}_+$, $\{x, y\}\mapsto d(x, y)$. This continuous function has a positive minimum $\varepsilon>0$ on the compact set $K$. It implies the $\varepsilon$-NEHE property for $G$.
\\
    \paragraph{\textit{Sufficiency direction}.}
Suppose that $G$ has the $\varepsilon$-NEHE property for some $\varepsilon>0$. Let $\xi\in \Lambda^{nc}G$ be a non-conical point.
By Proposition~\ref{lineNearOrbit}, the orbit $G\xi$ is $d_u$-discrete.
By Corollary~\ref{conicalOrBoundedParabolic}, the point $\xi$ is a bounded parabolic point.
Hence, $\Lambda G$ consists only of conical points and bounded parabolic points, i.e. the action is geometrically finite.

\section{Further Properties of Orbit Uniform Metrics}\label{chapter:global}

Corollary~\ref{conicalOrBoundedParabolic} characterizes the local behavior of the $d_u$-topology through expansive and non-expansive points, providing a foundation for the study of its global structure. We now apply this characterization to establish several global properties of the $d_u$-topology. In particular, we show that expansive points are $d_u$-dense in $\Lambda G$ by analyzing $d_u$-convergent sequences.

We first relate $d_u$-convergence to the convergence dynamics of the action. The following proposition shows that a non-trivial $d_u$-convergent sequence determines its attracting and repelling points uniquely and imposes strong restrictions on the behavior of all other points.

\begin{proposition}\label{generalDuConvergence}
    Suppose that a non-constant sequence $\{g_n\}$ in $G$ and points $\xi, \eta\in \Lambda^{ne}G$ satisfy $g_n\xi \duto \eta$. Then
    \begin{itemize}
        \item[(1)] $g_n^{-1}\eta\duto \xi$.
        \item[(2)] The attracting and repelling points of $\{g_n\}$ are $\eta$ and $\xi$, respectively.
        \item[(3)] For every point $\zeta\in M\setminus\{\xi\}$, the sequence $\{g_n\zeta\}$ $d_u$-diverges.
        \item[(4)] For every subsequence $\{g_{n_i}\}$ of $\{g_n\}$, the sequence $\{g_{n_i}^{-1}g_{n_{i+1}}\xi\}$ $d_u$-converges to $\xi$.
    \end{itemize}
\end{proposition}

\begin{proof}
We prove the four statements in order.
\begin{itemize}
    \item[(1)] Since the action of $G$ is isometric with respect to $d_u$, we have
    $$d_{u}(\xi, g_n^{-1}\eta)=d_u(g_n\xi, \eta)\to 0.$$

    \item[(2)] We show that every convergence subsequence has the same attracting and repelling points.
    Let $\{g_{n_i}\}$ be a convergence subsequence with attracting point $b\in M$ and repelling point $a\in M$.
    
    If $a\ne \xi$, then $\xi\in M\setminus\{a\}$ and hence $g_{n_i}\xi\to b$. Thus, $\eta=\lim\limits_{i\to\infty}g_{n_i}\xi=b$.
    
    If $a=\xi$, then, since $\xi$ is not conical, we must also have $b=\lim\limits_{i\to\infty}g_{n_i}\xi=\eta$.
    
    Thus, in either case, $b=\eta$.

    Applying the same argument to the inverse sequence $\{g_{n_i}^{-1}\}$, whose attracting and repelling points are $a$ and $b$, respectively, we obtain $a=\xi$.
    Since the convergence subsequence was chosen arbitrarily, every convergence subsequence of $\{g_n\}$ has attracting point $\eta$ and repelling point $\xi$. Consequently, $g_nx\to \eta$ locally uniformly for $x\in M\setminus\{\xi\}$.

    \item[(3)] Suppose that the sequence $\{g_n\zeta\}$ $d_u$-converges. By (2), its limit can only be $\eta$.
    Applying (2) to the pair $\zeta,\eta$, we see that $\zeta$ must be the repelling point of the convergence sequence $\{g_n\}$. Since this repelling point is $\xi$, it follows that $\zeta=\xi$.
    Hence $\{g_n\zeta\}$ $d_u$-diverges for every $\zeta\in M\setminus\{\xi\}$.

    \item[(4)] We have
    $$
    d_{u}(g_{n_i}^{-1}g_{n_{i+1}}\xi, \xi)=d_{u}(g_{n_{i+1}}\xi, g_{n_i}\xi)\to 0.
    $$
    Therefore, $g_{n_i}^{-1}g_{n_{i+1}}\xi\duto \xi$.
\end{itemize}
\end{proof}

Proposition~\ref{generalDuConvergence}(4) shows that every $d_u$-convergent sequence gives rise to a sequence of the form $g_n\xi\duto\xi$ for some $\xi\in M$. By Corollary~\ref{conicalOrBoundedParabolic}, such sequences exist if and only if $\xi$ is a non-expansive point. The following proposition shows that, when $\xi$ is non-expansive, such a sequence can be chosen to consist entirely of hyperbolic elements.

\begin{lemma}\label{hyperbolicElementsGconverge}
    Suppose that $\xi\in \Lambda G$ is a non-expansive point. Then there exists a sequence $\{g_n\}$ of hyperbolic elements in $G$ such that $g_n\xi\duto \xi$.
\end{lemma}

\begin{proof}
    By Corollary~\ref{conicalOrBoundedParabolic}, there exists a sequence $\{h_n\}$ in $G$ such that $h_n\xi\duto \xi$.
    By Lemma~\ref{producingShortHyperbolicElement}, there is a subsequence $\{h_{n_i}\}$ such that, for each fixed $i\in \mathbb{N}$, the element $h_{n_j}h_{n_i}$ is hyperbolic for all sufficiently large $j\in \mathbb{N}$.
    Hence we can find a sequence $\{g_i\}$ of the form $g_i=h_{n_j}h_{n_i}$ such that $n_j>n_i$ and each $g_i$ is hyperbolic.
    Therefore, 
    \begin{align*}
        d_{u}(g_i\xi, \xi)&= d_{u}(h_{n_j}h_{n_i}\xi, \xi)\\
        &\le d_{u}(h_{n_j}h_{n_i}\xi, h_{n_i}\xi)+d_u(h_{n_i}\xi, \xi)\\
        &=d_{u}(h_{n_j}\xi, \xi)+d_{u}(h_{n_i}\xi, \xi)\\
        &\to 0.
    \end{align*}
\end{proof}

The preceding proposition enables us to approximate every non-expansive point by fixed points of hyperbolic elements. We now exploit this construction to establish a global property of the $d_u$-topology. In particular, we show that expansive points are $d_u$-dense.

\begin{theorem}[=Theorem~\ref{expansiveDense}]
    The set $\Lambda^{e}G$ of conical points and bounded parabolic points is $d_u$-dense in $\Lambda G$.
\end{theorem}

\begin{proof}
Let $\xi\in \Lambda G$ be a non-expansive point.
By Lemma~\ref{hyperbolicElementsGconverge}, there exists a sequence $\{g_n\}$ of hyperbolic elements in $G$ such that $g_n\xi\duto\xi$.

Let $A=\{g_n:n\in \mathbb{N}\}$.
By the definition of the $d_u$-translation length, $\tau_u(g_n)\le d_u(\xi,g_n\xi)\to 0$.
By Corollary~\ref{translationLengthLowerBound}, any infinite subset of $A$ cannot have $\varepsilon$-NEHE property for any $\varepsilon>0$, i.e. $d_u(g_n^+,g_n^-)\to 0$.

\textit{Claim}: $g_n^+\duto\xi$.

Suppose, on the contrary, that $d_u(g_{n_i}^+,\xi)>\varepsilon$ for some subsequence $\{g_{n_i}\}$ and some $\varepsilon>0$.
By definition, there exists a sequence $\{h_i\}$ in $G$ such that $d(h_ig_{n_i}^+,h_i\xi)>\varepsilon$.
By passing to a subsequence, we may assume that $h_ig_{n_i}^+\to\eta$ and $h_i\xi\to\zeta$ for some distinct points $\eta,\zeta\in M$.
Since $d_u(g_{n_i}^+,g_{n_i}^-)\to 0$, the sequence $\{h_ig_{n_i}^-\}$ also converges to $\eta$.

Consider the sequence $\{h_ig_{n_i}h_i^{-1}\}$ of hyperbolic elements in $G$.
Note that $(h_ig_{n_i}h_i^{-1})^+=h_ig_{n_i}^+\to\eta$ and $(h_ig_{n_i}h_i^{-1})^-=h_ig_{n_i}^-\to\eta$.
By Lemma~\ref{hyperbolicElementFixPointConvergence}, $\{h_ig_{n_i}h_i^{-1}\}$ is a convergence sequence whose attracting and repelling points are both $\eta$.

Since $\lim\limits_{i\to\infty}h_i\xi=\zeta\ne\eta$, it follows that
$h_ig_{n_i}\xi=(h_ig_{n_i}h_i^{-1})(h_i\xi)\to\eta$.
However, $d(h_ig_{n_i}\xi,h_i\xi)\le d_u(g_{n_i}\xi,\xi)\to 0$, which implies that $h_ig_{n_i}\xi\to\zeta$, a contradiction.

Since $g_n^+$ is a conical point for every $n\in \mathbb{N}$, we conclude that $\xi$ belongs to the $G$-closure of $\Lambda^{e}G$.
\end{proof}

\bibliographystyle{amsalpha}
\bibliography{bibfile}

\end{document}